\documentclass[12pt]{amsart}

\usepackage{amsfonts}
\usepackage{amssymb}
\usepackage[letterpaper, left=2.5cm, right=2.5cm, top=2.5cm,
bottom=2.5cm,dvips]{geometry}
\usepackage{verbatim}
\usepackage[T1]{fontenc}
\usepackage{graphicx}
\usepackage{amsmath}

\usepackage{dsfont}

\usepackage{algorithm}
\usepackage[noend]{algpseudocode}

\usepackage[utf8]{inputenc}

\usepackage{bbm}
\usepackage{enumitem}

\newtheorem{theorem}{Theorem}
\theoremstyle{plain}

\newtheorem{conjecture}{Conjecture}

\newtheorem{lemma}{Lemma}

\newtheorem{proposition}{Proposition}

\numberwithin{equation}{section}

\usepackage{array,multirow}

\usepackage{xcolor}

\usepackage[backref]{hyperref}
\hypersetup{
	colorlinks,
	linkcolor={blue!60!black},
	citecolor={red!60!black},
	urlcolor={red!60!black}
}

\usepackage{tikz}

\begin{document}
	\title{Packing arithmetic progressions}
	
	\author{Noga Alon}
	\address{Department of Mathematics, Princeton University, Princeton, NJ 08544, USA and Schools of Mathematics and Computer Science, Tel Aviv University, Tel Aviv 69978, Israel}
	\email{nalon@math.princeton.edu}
	\thanks{The research of the first author was supported in part by NSF grant DMS-2154082}
	
	\author{Micha{\l} D\k ebski}
	\address{Faculty of Mathematics and Information Science, Warsaw University
		of Technology, 00-662 Warsaw, Poland}
	\email{michal.debski87@gmail.com}
	
	\author{Jaros\l aw Grytczuk}
	\address{Faculty of Mathematics and Information Science, Warsaw University
		of Technology, 00-662 Warsaw, Poland}
	\email{jaroslaw.grytczuk@pw.edu.pl}
	\thanks{The third author was supported in part by Narodowe Centrum Nauki, grant 2020/37/B/ST1/03298.}
		
	\author{Jakub Przyby\l o}
	\address{AGH University of Krakow, al. A. Mickiewicza 30, 30-059 Krakow, Poland}
	\email{jakubprz@agh.edu.pl}
	\thanks{The fourth author was supported in part by the AGH University of Krakow, grant no. 16.16.420.054, funded by the Polish Ministry of Science and Higher Education.}

\begin{abstract}
	Let $\mathcal{F}=\{A_1,A_2,\ldots,A_k\}$ be a collection of finite arithmetic progressions, where each $A_d$ is an initial segment of the set $D_d=\{d,2d,3d,\ldots\}$ of consecutive multiples of a positive integer $d$. Let $m(\mathcal{F})$ denote the minimum length of an interval containing pairwise disjoint \emph{shifted} copies of all members of the family $\mathcal{F}$. 
	We study this parameter in the following two cases: for a fixed positive integer $n$, (1) each progression in $\mathcal{F}$ has the form $A_d=D_d\cap\{1,2,\ldots,n\}$, and (2) all progressions $A_d$ of $\mathcal{F}$ have the same size $n$, that is, $A_d=D_d\cap \{1,2,\ldots, nd\}$. We in particular derive the following asymptotic estimates. In case (1), when $k=n$, we get $m(\mathcal{F})=\Theta(n^{3/2}/\ln n)$. In case (2), when $k=n$, we get $m(\mathcal{F})=\Theta(n^3/\ln n)$, while if $k>k_0(n)$, then $m(\mathcal{F}) < 3kn$. In both cases we additionally determine $m(\mathcal{F})$ asymptotically or settle its order of magnitude for all $k<n$.
\end{abstract}

	\maketitle

	\section{Introduction}
For any positive integer $m\in \mathbb{N}$ 
put $[m]=\{1,2,\ldots,m\}$. Let $D_d=\{d,2d,3d,\ldots\}$ denote the set of all positive multiples of a number $d\in \mathbb{N}$. For a given subset $A\subseteq \mathbb{N}$, its \emph{shifted} copy 
is any set of the form $s+A=\{s+a:a\in A\}$, where $s$ is an arbitrary integer. We will use the term \emph{shift} to refer to both the set $s+A$ and the number $s$.

Let $\mathcal{F}=\{A_1,A_2,\ldots,A_k\}$ be a collection of finite subsets of $\mathbb{N}$. We say that $\mathcal{F}$ can be \emph{packed} into $[m]$ if there exist integers $s_i$ such that:
\begin{enumerate}
	\item [(1)]$s_i+A_i\subseteq [m]$, for each $i \in [k]$,
	\item [(2)] $(s_i+A_i)\cap (s_j+A_j)=\emptyset$, for all $i,j \in [k]$, $i\neq j$.
\end{enumerate}

Let $m(\mathcal{F})$ denote the least possible $m\in \mathbb{N}$ such that $\mathcal{F}$ can be packed into $[m]$. Our goal is to estimate the parameter $m(\mathcal{F})$ for families consisting of 
arithmetic progressions whose differences are pairwise distinct and form the set $[k]$.
Obviously, the problem makes sense only subject to imposing restrictions 
to the lengths of the progressions. We consider two natural variants here. 

Firstly, we investigate the bounded \emph{diameter} setting, that is for a fixed $n\in \mathbb{N}$, we consider families $\mathcal{F}=\{A_1,A_2,\ldots,A_k\}$, 
of maximal initial segments of the sets $D_d$ entirely contained in $[n]$.
More formally, each member of the family has the form $A_d=D_d\cap [n]=\{id:1\leqslant i\leqslant \lfloor n/d\rfloor\}$. For any subset $D\subseteq [n]$, let us denote $\mathcal{F}_D=\{A_d:d\in D\}$, and let $m_D(n)=m(\mathcal{F}_D)$. If $D=[k]$, with $k\leqslant n$, then we simply write $m_k(n)=m(\mathcal{F}_{[k]})$ and $m(n)=m_n(n)$. Notice that for $k>n$, the paramater $m_k(n)$ is not defined.

In our first result we determine the asymptotic order of the function $m(n)$.

\begin{theorem} \label{t11}
	We have $m(n)=\Theta(n^{3/2}/\ln n)$. More precisely,
		$$\left(\frac43-o(1)\right) \frac{n^{3/2}}{\ln n} \leqslant  m(n) \leqslant \left(\frac53+o(1)\right)
		\frac{n^{3/2}}{\ln n}.$$
\end{theorem}

It would be nice to close the gap between these bounds, and we suspect that the constant $4/3$ is closer to the truth. Actually, the above lower bound holds even if we restrict the differences of arithmetic progressions to prime numbers not exceeding $\sqrt{n}$. On the other hand, if $k\leqslant \sqrt{n}$, then $m_k(n)\leqslant (1.52+o(1))n^{3/2}/\ln n$, while for $k\ll\sqrt{n}$, we have $m_k(n) = (1+o(1))kn/\ln k$, which is a certain indication in favor of the supposition (see Theorem \ref{t31}).

Secondly, we consider similar collections of initial segments of sets of multiples $D_d$ having the same \emph{size} $n$. More precisely, we investigate the packing problem for families $\mathcal{G}=\{B_1,B_2,\ldots, B_k\}$, where $B_d=D_d\cap [nd]=\{id:1\leqslant i\leqslant n\}$, with $d\in [k]$. As before, for any subset $D$ of positive integers, we denote $\mathcal{G}_D=\{B_d:d\in D\}$. Also, the corresponding parameters are defined in the same way, and we denote them analogously by $M_D(n)=m(\mathcal{G}_{D})$, $M_k(n)=m(\mathcal{G}_{[k]})$, and $M(n)=m(\mathcal{G}_{[n]})$.

Our next result determines the asymptotic order of the function $M(n)$.

\begin{theorem} \label{t12}
	We have $M(n)=\Theta(n^3/\ln n)$. More precisely,
	$$\left(\frac16-o(1)\right) \frac{n^{3}}{\ln n} \leqslant  M(n) \leqslant \left(0.526+o(1)\right)
	\frac{n^{3}}{\ln n}.$$
\end{theorem}

In section~\ref{SectionMk-small_k} we formulate a more detailed version of the theorem 
with estimates for $M_k(n)$ when $k\leqslant n$ (see Theorem \ref{t41}).

Notice that the parameter $M_k(n)$ is perfectly viable for $k>n$, unlike $m_k(n)$. 
For $k$ much larger than $n$, by using different techniques, 
we get the following result.

\begin{theorem} \label{t13}
	For every $n\in \mathbb{N}$ there exists $k_0=k_0(n)$ such that, for all $k\geqslant k_0$, we have 
$$M_k(n) \leqslant 3nk.$$
\end{theorem}

Let us remark that a similar statement may be far from true for arbitrary families of $n$-element sets (with bounded diameter). Consider, for instance, the set $$S_r=\{1,2,\ldots, r, 2r, 3r, \ldots, (r+1) r\}$$ of size $|S_r|=2r=n$. It is not hard to see that there is no way to pack two shifted copies of $S_r$ so that they overlap. Hence, packing of $k$ copies of $S_r$ requires an interval of length at least $kr(r+1)=\Omega(kn^2)$. Similar examples can be constructed with reference to other results of this paper.

The problem of packing arithmetic progressions with all gaps distinct emerged in connection to certain puzzles having a common flavor of searching for perfect matchings in hypergraphs defined by some distance constraints. One of them is the following problem of \emph{barricade building} posed by Cipra and popularized by Guy \cite{Guy}.

Let $S_n$ denote the set of all permutations of $[n]$. Two permutations, $(a_1,a_2,\ldots,a_n)$ and $(b_1,b_2,\ldots,b_n)$, are \emph{dependent} if $a_1+\cdots+a_i=b_1+\cdots+b_j$, for some $i,j \in [n-1]$. A subset $A\subseteq S_n$ is \emph{independent} if no two of its elements are dependent. What is the maximum size of an independent set in $S_n$? An easy upper bound is $\lfloor n/2\rfloor+1$ and a natural conjecture is that it is attained. In \cite{Debski GPPS-N} we proved that it is at least linear by using initially some packing arguments, like the ones investigated here.

Another, more direct inspiration came from the puzzle of \emph{perfect rhythmic tilings}, invented by Johnson \cite{Johnson} (see also \cite{Amiot} and \cite{Sloane}). The task is to partition a segment $[nk]$ into $k$ arithmetic progressions, each with $n$ terms, but no two with the same gap. Is it true that, for every fixed $n$, such tilings exist for all large enough $k$? The answer is not known even for $n=3$.

Let us finally mention that structures of our interest occur also in the context of some more famous problems. For instance, the \emph{multiplication table} problem of Erd\H{o}s \cite{Erdos} (see \cite{Ford}) asks simply for the cardinality of the union $B_1\cup \cdots\cup B_n$ of arithmetic progressions defined above, while the \emph{Arithmetic Kakeya Conjecture} of Katz and Tao \cite{Katz-Tao} (stated independently by Ruzsa \cite{Ruzsa}), concerns the least possible size of the union $(s_1+B_1)\cup \cdots\cup (s_k+B_k)$ of their shifted copies. Even Riemann Hypothesis can be equivalently formulated by using a (slightly modified) \emph{incidence matrix} of the family $\{A_1,A_2,\ldots, A_n\}$, known as the \emph{Redheffer matrix} \cite{Redheffer}, whose determinant coincides with the partial sum of the M\"{o}bius function (see \cite{Broughan}).

Proofs of the above stated results are presented in the three subsequent sections. The last section contains discussion together with some further simple results, as well as several open problems.

Throughout the paper we use the standard asymptotic notation. In particular, we assume that values of $n$ or $k$ are sufficiently large whenever this is needed. The usual notation $o(1)$ denotes a quantity that tends to $0$ as the relevant parameter tends to infinity. 
Should this not be clear from context we will specify such parameter $p$ by writing $o_p(1)$ instead.
The logarithm in the natural base $e$ is denoted by $\ln x$, and $\log x$ denotes logarithm in base $2$.

\section{Bounds for $m(n)$; Proof of Theorem \ref{t11}}
\subsection{Preliminaries}
We examine packings of integer sets, hence an interval will be identified with integer points, sometimes called \emph{elements}, within it -- the number of these will be called the \emph{length} of an interval.
We will exploit several known estimates and will be mostly focused on asymptotic approximations. For definiteness of notation let us however assume that an interval of length $x$, where $x$ is a real number, is assumed to contain $\lfloor x\rfloor$ (integer) elements, unless explicitly stated otherwise. We say an integer set $A$ \emph{fits in} an interval $I$ when there is its shifted copy $s+A$ entirely contained in $I$.

We will be in particular repeatedly using the following well-known consequences of the Prime Number Theorem.
\begin{lemma} \label{l21}
	The sum of all primes up to $x$ is
	$$
	(1+o(1))\frac{x^2}{2 \ln x},
	$$while the sum of squares of all primes up to $x$ is
	$$
	(1+o(1))\frac{x^3}{3 \ln x}.
	$$
\end{lemma}

Let us further recall that an integer is \emph{$y$-smooth} if its largest prime divisor is at most $y$. We will need the following result of Canfield, Erd\H{o}s and Pomerance \cite{CEP}, see also \cite{Bru}. 
\begin{lemma}[\cite{CEP}]
	\label{l22}
	If $x=y^u$ where $u$ tends to infinity with $x$, and $y \geqslant (\ln x)^{1.001}$, then the number of $y$-smooth integers that do not exceed $x$ is 
	$$\frac{x}{u^{(1+o(1))u}},$$
	where the $o(1)$-term tends to $0$ as $x$ (and also $y$ and $u$) tend to infinity.
\end{lemma}
We will also use the following simple fact.
\begin{lemma}
	\label{l23}
	Let $A$ and $B$ be two integer subsets of an interval $I$ of length $g>n$, and assume that $B$ fits in an interval of length $n$. Suppose, further, that $|A| \cdot |B| < g-n$. Then there is a shifted 	copy $z+B$ of $B$ which is disjoint from $A$ and fits in $I$.
\end{lemma}
\begin{proof}
	There are at least $g-n$ shifted copies of $B$ that lie inside $I$. At most $|A| \cdot |B|$ of these may intersect $A$, since each of them is obtained by shifting some element $b \in B$ to some $a \in A$. Therefore, there is a shifted copy of $B$  satisfying the required properties.
\end{proof}

\subsection{The lower bound in Theorem \ref{t11}}
The proof of the lower bound is quite straightforward, showing that it holds even for $m_D(n)$ where $D$ is the set of all primes $p \leqslant \sqrt n$.

\begin{proposition}
	\label{p24}
	Let $D$ be the set of all primes that do not exceed $\sqrt n$.
	Then 
	$$m_D(n) \geqslant \left(\frac{4}{3}-o(1)\right)\frac{n^{3/2}}{\ln n}.$$
\end{proposition}
\begin{proof}
	Let $D$ be as above, and let $s_d$ denote the shift of $A_d$ so that the progressions $s_d+A_d, d \in D$, are pairwise disjoint and fit in $[m]$. We have to show that 
	$$
	m \geqslant \left(\frac{4}{3}-o(1)\right)\frac{n^{3/2}}{\ln n}.
	$$
	
	Since for any  two distinct primes, $p$ and $q$, the progressions $s_p+A_p$ and $s_q+A_q$ are disjoint,
	the Chinese Remainder Theorem implies that 
	the shortest  interval
	containing $s_p+A_p$ and the shortest interval containing $s_q+A_q$ cannot have more than $pq$ common elements. This implies that $|s_p-s_q| \geqslant n-\max\{p,q\}-pq$.
	
	Let $p_1<p_2< \ldots <p_t$ be all elements of $D$, ordered from small to large, and let $\pi$ be
	a permutation of $[t]$ defined according to the order of the shifts $s_d$, that is,
	$$s_{p_{\pi(1)}} \leqslant s_{p_{\pi(2)}} \leqslant \ldots \leqslant 
	s_{p_{\pi(t)}}.$$
	By the above fact and Lemma~\ref{l21} it follows that the difference between $s_{p_{\pi(t)}}$ and $s_{p_{\pi(1)}}$
	is at least $$(1-o(1)) nt -\sum_{i <t} p_{\pi(i)} p_{\pi(i+1)}.$$
	By the simple 	rearrangement inequality (see, e.g., \cite{HLP}), the last sum is at most the sum of squares of all elements in $D$.
	 By Lemma \ref{l21} this sum is 
	$$(1+o(1)) \frac{(\sqrt n)^3}{3 \ln (\sqrt n)}.$$ 
	Since $t=(1+o(1)) \sqrt n/ \ln (\sqrt n)$ by the Prime Number Theorem, the desired result follows.
\end{proof}

\subsection{The upper bound in Theorem \ref{t11}}

The proof of the upper bound in Theorem \ref{t11} is more complicated 
than that of the lower bound.
It will be convenient to describe the applied packing procedure 
in several steps, as follows. Put $s= \sqrt n (\ln n)^3$. To simplify the presentation, we omit 
floor and ceiling signs here, when they are not essential.

\vspace{0.2cm}

\noindent
{\bf Step 1: differences $d \leqslant s$ which are $(\ln n)^5$-smooth.}
\vspace{0.1cm}

\noindent
By Lemma \ref{l22} the number of such differences is $s^{4/5+o(1)} =n^{2/5+o(1)}$. 
We embed all progressions $A_d$ with differences $d$ considered in this step in 
separate (consecutive) intervals of length $n$.
The total length of these intervals is thus 
$n^{7/5+o(1)}=o(n^{3/2}/ \ln n)$.
\vspace{0.2cm}

\noindent
{\bf Step 2: differences $d \leqslant s$ with a prime divisor $p$ satisfying
	$(\ln n)^5 \leqslant p \leqslant \sqrt s$.}
\vspace{0.1cm}

\noindent
For each prime $p$ in the above range, define $$D_p^{(s)}=D_p\cap [s]=\{d \leqslant s: p| d\}.$$ Note that $|D_p^{(s)}| \leqslant s/p$. Note also that for each collection of $p$ consecutive elements in $D_p^{(s)}$, the corresponding progressions $A_{pi},A_{p(i+1)}.\ldots,A_{p(i+p)}$ can be packed in an interval of length $n$. Indeed, each $A_{pj}$ may occupy positions congruent to $j\pmod p$ in $[n]$. Therefore, the progressions with differences in $D_p^{(s)}$ can be packed in an interval of total length of at most $n \lceil |D_p^{(s)}|/p \rceil \leqslant O(ns/p^2)$.

Summing over all $p$ in the above range gives at most
$$
O\left(ns \sum_{p \geqslant (\ln n)^5} \frac{1}{p^2}\right) \leqslant
O\left(ns \frac{1}{(\ln n)^5}\right) 
= O\left(\frac{n^{3/2}}{(\ln n)^2}\right)
=o\left(\frac{n^{3/2}}{\ln n}\right).
$$
\vspace{0.2cm}

\noindent
{\bf Step 3: differences $d \leqslant s$ with a prime divisor $p$ satisfying $\sqrt s \leqslant p \leqslant s$.}
\vspace{0.1cm}

\noindent
Note that here we may also assume that these differences have no prime divisor larger than $(\ln n)^5$ and smaller than $\sqrt s$, but this is not needed for the argument.

	Consider first the differences $d \leqslant s=\sqrt n (\ln n)^3$ divisible by a prime $p>\sqrt n$. 
	Note that all progressions with differences equal to such primes can be packed in one interval of length $n$. Indeed, for each such $p$ it suffices to place the corresponding progression $A_p$ in $[n]$ at positions congruent to $0\pmod p$. 
	By the Chinese Remainder Theorem, the resulting embeddings in $[n]$ of any two such progressions 
	$A_p$ and $A_q$, with primes $p>q>\sqrt n$, cannot intersect, as desired.
	Likewise, we may pack all progressions $A_{2p}$, with prime $p>\sqrt n$, in a single (new) interval of length $n$, and so on. As all $d\leqslant s$ divisible by   a prime $p>\sqrt n$ are of the form $pg$, where 
	$g\leqslant (\ln n)^3$, we can fit all the corresponding progressions $A_d$ disjointly in an interval of total length $n(\ln n)^3=o(n^{3/2}/\ln n)$.
	
	Secondly, note that for each fixed prime $p$ such that $\sqrt s\leqslant p \leqslant \sqrt n$, all progressions $A_d$ with differences $d$ in the set $D_p^{(s)}$ can be packed in one interval of length $n$, since $|D_p^{(s)}| \leqslant s/p \leqslant p$. As we will expose that the remaining progressions, considered in Step 4 below, require no additional space, this suffices to obtain the upper bound with multiplicative factor roughly $2$ rather than $5/3$ (or $4/3$). Indeed, the number of such primes is at most $(1+o(1)) \sqrt n/ \ln (\sqrt n)$, by the Prime Number Theorem, so the total length of an interval required for these is at most $(2+o(1))n^{3/2}/ \ln n$. 
	We will however be more careful here and provide a tighter packing, admitting overlaps of minimal intervals containing $A_d$ with $d$ in selected sets $D_p^{(s)}$.

	First note that, by the above remarks, the progressions with differences in all sets $D_p^{(s)}$ with $\sqrt s\leqslant p \leqslant \sqrt n/\ln n$, can be packed disjointly in an interval of length $o(n^{3/2}/\ln n)$, due to the Prime Number Theorem, which implies that the number of such primes $p$ is $o(\sqrt{n}/\ln n)$. We are thus left with sets $D_p^{(s)}$ where $p$ are primes in the range $\sqrt n/\ln n < p \leqslant \sqrt n$. Set $r=\lfloor(\ln n)^4\rfloor$.
	
	Consider two consecutive primes $p<q$ such that $\sqrt n/\ln n< p<q \leqslant \sqrt n$ and note that both sets $|D_p^{(s)}|$ and $|D_q^{(s)}|$ are of size at most $r$, since $s=\sqrt n (\ln n)^3$. We will demonstrate that one may disjointly pack all progressions with differences in $D_p^{(s)}\cup D_q^{(s)}$ in an interval of length $2n-(1-o(1))pq$. It suffices to show, for some fixed integer $m=(1-o(1))pq$, that we may find $2r$ disjoint arithmetic progressions such that each among the first half of these includes every $p$-th element in $[m]$, while every remaining one consists of every $q$-th element of $[m]$. In other words, we will prove that there are $r$ distinct residues $i\pmod p$ and $r$ distinct residues $j \pmod q$ such that for any such pair $(i,j)$ and every $\ell\in [m]$, either $\ell\not\equiv i\pmod p$ or $\ell\not\equiv j\pmod q$. If so, the corresponding progressions can be extended (e.g., to the left for difference $p$ and to the right for difference $q$) to provide disjoint packings of $r$ arithmetic progressions of size $\lfloor n/p\rfloor$ and difference $p$ and $r$ arithmetic progressions of size $\lfloor n/q\rfloor$ and difference $q$ in an interval of length $2n-(1-o(1))pq$,
	which is more than enough for our purpose.
	
	We claim that we may take $m=(p-2r)q$ and $i=(-i'q) \pmod p$, with $i'\in[r]$, and $j=(-j'p) \pmod q$, with $j'\in[r]$ (hence, $i\in [p-1]$ and $j\in [q-1]$). Indeed, suppose that for some pair $(i,j)$ and any given $\ell\in [m]$, we have $\ell\equiv i\pmod p$ and $\ell\equiv j\pmod q$. Then there exist integers $s\in [0,q-2r)$ and $t\in [0,p-2r)$ such that $\ell = sp+i$ and $\ell = tq+j$,
	and there are positive integers $i',j'\in[r]$, $t'\leqslant r$, $s'$ such that $i=-i'q+s'p$ and $j=-j'p+t'q$. Consequently, $(s+s'+j')p=(t+t'+i')q$, hence $t+t'+i'\geqslant p$, and thus $t\geqslant p-2r$, a contradiction.
	
	Let $p_1>p_2>\ldots>p_h$ be all prime numbers in the interval $(\sqrt n/\ln n, \sqrt n]$. By pairing every two consecutive primes of this collection and using the above overlapping packing in a separate interval for every such pair, we may pack disjointly all progressions with differences in $D_{p_1}\cup\ldots\cup D_{p_h}$ in an interval of length at most
	\begin{eqnarray}
		&&(1+o(1))n\frac{\sqrt{n}}{\ln (\sqrt{n})}-(1-o(1))\sum_{1\leqslant i \leqslant \lfloor h/2\rfloor}\left(p_{2i}\right)^2 \nonumber\\
		&=& 
		(1+o(1))n\frac{\sqrt{n}}{\ln (\sqrt{n})}-(1-o(1))\frac{1}{2}\cdot \frac{(\sqrt{n})^3}{3\ln(\sqrt{n})} \nonumber\\
		&=& \left(\frac{5}{3}-o(1)\right)\frac{n^{3/2}}{\ln n}, \label{tighter_packing}
	\end{eqnarray}
	due to the Prime Number Theorem and Lemma~\ref{l21}.

\vspace{0.2cm}

\noindent
{\bf Step 4: differences $d > s$.}
\vspace{0.1cm}

\noindent
It is easy to see that the total size of all progressions in the investigated family $\mathcal{F}=\{A_1,A_2,\ldots, A_n\}$ is only
$O(n \ln n)$, while each of the remaining unpacked progressions is of size at most $\sqrt n/ (\ln n)^3$. Therefore all of them can be packed one by one into the existing packing using Lemma \ref{l23} (with room to spare).

This completes the proof of the upper bound and hence the proof of Theorem \ref{t11}. \hfill $\Box$

	\section{Bounds for $m_k(n)$ with small $k$}
	
	We will now be concerned with packing only the first $k< n$ arithmetic progressions of the collection $A_1,A_2,\ldots, A_n$ in a short interval. 
	We will prove the following supplementary result to Theorem~\ref{t11} for $k\leqslant \sqrt{n}$. 
	(For  $k>\sqrt{n}$, the thesis of Theorem~\ref{t11} remains valid and provides the right order of $m_k(n)$, since the corresponding lower bound exploits prime differences up to $\sqrt{n}$ exclusively.)

	\begin{theorem}\label{t31}
		Let $k$ and $n$ be arbitrary positive integers, with $ k\leqslant \sqrt{n}$. If $k\ll \sqrt{n}$, then 
		$$m_k(n) = (1+o_k(1))\frac{kn}{\ln k}.$$
		Otherwise, if $k\leqslant \sqrt{n/2}$, then
		$$\left(1-\frac{k^2}{3n}-o(1)\right)\frac{kn}{\ln k} \leqslant m_k(n)\leqslant \left(1-\frac{k^2}{3n}+\frac{k^2}{48n}+o(1)\right)\frac{kn}{\ln k},$$
		and if $\sqrt{n/2} \leqslant k \leqslant \sqrt{n}$, then
		$$\left(1-\frac{k^2}{3n}-o(1)\right)\frac{kn}{\ln k} \leqslant m_k(n)\leqslant \left(\frac{1}{2}+\frac{1}{3\sqrt{2}}\frac{\sqrt{n}}{k}+\frac{k^2}{48n}+o(1)\right)\frac{kn}{\ln k}.$$
	\end{theorem}

	\begin{proof} 
		For the lower bound it again suffices to focus on progressions $A_p$ with prime $p\leqslant k$. By the Prime Number Theorem, the number of such primes equals $t=(1+o_k(1))k/\ln k$. 
		By the same reasoning as in the proof of Proposition~\ref{p24}, 
		since intervals including arithmetic progressions with differences given by the analyzed primes cannot overlap too much, any interval including disjoint copies of all these progressions cannot be shorter than
		$$\left(1-o_k(1)\right)nt-\frac{k^3}{3\ln k},$$
		which concludes the proof of the lower bound in all cases.

		For the upper bound, note first that by Lemma~\ref{l22}, the number of $(\ln k)^2$-smooth positive integers $d\leqslant k$ is at most $k^{1/2+o(1)}$. So, the arithmetic progressions $A_d$ with such differences can be packed in an interval of length 
		\begin{equation}\label{PackingInitial1}
		o_k\left(\frac{kn}{\ln k}\right).
		\end{equation}
		
		For a prime $p$, recall that $D_p^{(k)}=D_p\cap [k]=\{d\leqslant k: p|d \}$. Analogously as in the proof of Theorem~\ref{t11}, 
		for every fixed $p\leqslant \sqrt{k}$, any $p$ progressions $A_d$ with $d\in D_p^{(k)}$ can be packed in $[n]$. Thus, all arithmetic progressions with differences in $D_p^{(k)}$ for all primes $p\in [(\ln k)^2, \sqrt{k}]$ can be packed in an interval of total length at most
		\begin{equation}\label{PackingInitial2}
		n\cdot\sum_{(\ln k)^2\leqslant p \leqslant \sqrt{k}}\left\lceil\frac{|D_p^{(k)}|}{p}\right\rceil \leqslant nk\cdot O\left(\sum_{p\geqslant (\ln k)^2}\frac{1}{p^2}\right) = O\left(\frac{kn}{(\ln k)^2}\right) = o_k\left(\frac{kn}{\ln k}\right).
		\end{equation}

		We are left with differences in $D_p^{(k)}$ for primes satisfying $\sqrt{k}<p\leqslant k$.
		As before, 
		the progressions $A_d$ with differences $d$ in any such fixed set $D_p^{(k)}$ can be packed in $[n]$.
		Thus, by the Prime Number Theorem, all progressions $A_d$ with $d$ in the union of the sets $D_p^{(k)}$ over $p\leqslant k/\ln k$ can be packed in an interval of total length at most 
		\begin{equation}\label{PackingInitial3}
		O\left(\frac{nk}{(\ln k)^2}\right).
		\end{equation}
		
		Likewise, the leftover progressions $A_d$ with $d\in D_p^{(k)}$ for the remaining (large) values of $p$, with  $k/\ln k<p\leqslant k$, can be packed in an interval of length 
		\begin{equation}\label{WithoutOverlappings}
		n(1+o_k(1))\frac{k}{\ln k}
		\end{equation}
		so that the shortest intervals containing the embeddings of progressions $A_d$ with $d$ in any $D_p^{(k)}$ do not overlap with the shortest interval containing the embeddings of  progressions $A_d$ with $d$ in any other $D_{p'}^{(k)}$, $p'\neq p$. We may however tighten such packing by arranging these intervals according to the decreasing order of the corresponding primes and assure their substantial overlappings.
		
		For primes $p\in (k/\ln k,k/2]$ we proceed as in Step 3 of the proof of Theorem~\ref{t11}, i.e., we first partition them into pairs of consecutive primes and for any such pair $\{p,q\}$ we assure an overlapping of length $(1-o_k(1))pq$. By Lemma~\ref{l21}, the total sum of the lengths of such overlappings is thus
		\begin{equation}\label{FirsttOverlappings}
		(1-o_k(1))\frac{1}{2}\cdot \frac{\left(\frac{k}{2}\right)^3}{3\ln \left(\frac{k}{2}\right)} 
		\end{equation}
		cf. estimates in~(\ref{tighter_packing}).
		
		Let $p_1>p_2>\ldots>p_z$ be all (remaining) prime numbers in $(k/2,k]$. Since for each such $p_i$ 
		we  have $D_{p_i}^{(k)}=\{p_i\}$, we may be more efficient in this regime and provide two-sided overlappings for shits of (almost) all $A_{p_i}$. Namely, by the Chinese Remainder Theorem, it is straightforward to 
		 assure that the intersection of any two intervals containing shifted copies of $A_{p_i}$ and $A_{p_{i+1}}$ has length at least $(1-o_k(1))\min\{p_ip_{i+1},n/2\}$, for every $i<z$.
		 Estimating again by means of  Lemma~\ref{l21}, we thus infer that for $k\leqslant \sqrt{n/2}$
		  all the remaining arithmetic progressions, those not taken into account in~(\ref{PackingInitial1}), (\ref{PackingInitial2}) and~(\ref{PackingInitial3}),  can be packed in an interval not exceeding the following length (where the first two components below correspond to~(\ref{WithoutOverlappings}) and~(\ref{FirsttOverlappings}), while the last one estimates the total length of overlappings for primes in $(k/2,k]$):		
		\begin{eqnarray}
			&&n(1+o_k(1))\frac{k}{\ln k}
			-(1-o_k(1))\frac{1}{2}\cdot \frac{\left(\frac{k}{2}\right)^3}{3\ln \left(\frac{k}{2}\right)} 
			-(1-o_k(1))\left(\frac{k^3}{3\ln k} - \frac{\left(\frac{k}{2}\right)^3}{3\ln \left(\frac{k}{2}\right)} \right) \nonumber\\
			&=& (1-o_k(1))\left(\frac{nk}{\ln k} - \frac{k^3}{\ln k} \left(\frac{1}{3}-\frac{1}{48}\right)\right); \nonumber
		\end{eqnarray}
		while for  $k\in [\sqrt{n/2}, \sqrt{n}]$ -- in an interval 
		not exceeding the following length (where the last two components
		estimate the total lengths of overlappings separately for primes in $(k/2,\sqrt{n/2}]$ and in $(\sqrt{n/2},k]$):
		\begin{eqnarray}
					&&(1+o(1))\left(\frac{nk}{\ln k}
			-\frac{1}{2}\cdot \frac{\left(\frac{k}{2}\right)^3}{3\ln \left(\frac{k}{2}\right)} 
			-\left(\frac{\left(\frac{\sqrt{n}}{\sqrt{2}}\right)^3}{3\ln \left(\frac{\sqrt{n}}{\sqrt{2}}\right)} - \frac{\left(\frac{k}{2}\right)^3}{3\ln \left(\frac{k}{2}\right)} \right) 
			-\frac{n}{2}\cdot\left(\frac{k}{\ln k}-\frac{\left(\frac{\sqrt{n}}{\sqrt{2}}\right)}{\ln \left(\frac{\sqrt{n}}{\sqrt{2}}\right)}\right)\right)\nonumber\\
			&=& (1+o(1))\left(\frac{nk}{2\ln k} + \frac{n^{3/2}}{3\sqrt{2}\ln k} +\frac{k^3}{48\ln k}\right). \nonumber
		\end{eqnarray}
	\end{proof}

	\section{Bounds for $M_k(n)$ when $k\leqslant n$}\label{SectionMk-small_k}
	
	Let $n$ be a fixed positive integer. We now turn our attention to packing collections of arithmetic progressions $\{B_1,B_2,\ldots, B_k\}$, where each $B_d$ has size $n$ and is of the form $B_d=\{d,2d,\ldots,nd\}$, for $d=1,2,\ldots,k$. We will first study the case of $k\leqslant n$. 
	
	We will need the following well-known number theoretic facts.
	\begin{lemma}\label{l41}
		For every integer $r\geqslant 2$, $\sum_{i=1}^{\infty}n^{-r}=\zeta(r)$, where $\zeta$ is the Rieman zeta function. In particular,
		$$\zeta(2)=\frac{\pi^2}{6}\approx 1.645,~  \zeta(3)\approx 1.202.$$
	\end{lemma}

Recall that $M_k(n)$ denotes the least size of an interval in which the above collection of progressions $B_d$ can be packed. Here is an analog of Theorem \ref{t31} for the function $m_k(n)$.

	\begin{theorem}\label{t41}
		Let $k\leqslant n$ be arbitrary positive integers. If $k\ll n$, then 
		$$\left(\frac{1}{2}-o_k(1)\right)\frac{k^2n}{\ln k}\leqslant M_k(n) \leqslant \left(\frac{\zeta(2)}{2}+o_k(1)\right)\frac{k^2n}{\ln k}
		<\left(0.823+o_k(1)\right)\frac{k^2n}{\ln k}.$$
		Otherwise, if $k\leqslant n/2$, then 
		$$\left(\frac{1}{2}-\frac{k}{3n}-o(1)\right)\frac{k^2n}{\ln k} \leqslant M_k(n)\leqslant \left(\frac{\zeta(2)}{2}-\frac{\zeta(3)k}{3n}+o(1)\right)\frac{k^2n}{\ln k},$$
		while if $k\in[n/2,n]$, then
		$$\left(\frac{1}{2}-\frac{k}{3n}-o(1)\right)\frac{k^2n}{\ln k} \leqslant M_k(n)
		\leqslant \left(\frac{\zeta(2)-0.5}{2}-\frac{(\zeta(3)-1)k}{3n}
		+\frac{1}{48}\cdot\frac{n^2}{k^2}
		+o(1)\right)\frac{k^2n}{\ln k}.$$
	\end{theorem}
	\begin{proof} 
		The lower bound follows by analysis parallel 
		to the ones in the proofs of Theorems~\ref{t11} and~\ref{t31}, but this time we estimate the minimal length of an interval containing disjoint copies of all $B_p$ with prime $p\leqslant k$. 
		Let $p_1<p_2< \ldots <p_t$ be all such primes, and let $\pi$ be
	a permutation of $[t]$ defined consistently with 
	the order of shifts $s_{p_i}$ of all $B_{p_i}$ yielding optimal packing of these in the shortest interval, say $I$.
	By the Chinese Remainder Theorem, $|s_{\pi(p_{i+1})}-s_{\pi(p_i)}|\geqslant (n-1)p_{\pi(i)} - p_{\pi(i)} p_{\pi(i+1)}$ for every $i<t$. By Lemma \ref{l21}, we thus infer that the length of $I$ equals at least 
	$$\sum_{i <t} (n-1)p_{\pi(i)} - \sum_{i <t} p_{\pi(i)} p_{\pi(i+1)} 
	\geqslant (1-o_k(1))\left(\frac{nk^2}{2\ln k} - \frac{k^3}{3\ln k}\right).$$

		For the upper bound we once more first note that by Lemma~\ref{l22}, the number of $(\ln k)^2$-smooth positive integers $d$ not exceeding $k$ equals at most  $k^{1/2+o(1)}$, so the corresponding progressions $B_d$ can be packed into an interval of length at most $k^{1/2+o(1)}\cdot kn = o_k(k^2n/\ln k)$.
		
		As before, we define $D_p^{(k)}=\{d\leqslant k : p | d\}$, and likewise as in equation (\ref{PackingInitial2}), for primes $p\in ((\ln k)^2,\sqrt{k}]$ we can do with (a fresh) interval of length
		$$nk\cdot\sum_{(\ln k)^2\leqslant p\leqslant \sqrt{k}}\left\lceil\frac{|D_p^{(k)}|}{p}\right\rceil 
		= O\left(\frac{k^2n}{(\ln k)^2}\right)
		= o_k\left(\frac{k^2n}{\ln k}\right).$$
		
		Next, by the Prime Number Theorem, there are at most $(1+o_k(1))k/(\ln k)^2$ primes $p\in(\sqrt{k},k/\ln k]$. As for each of these it is sufficient to use an interval of length $kn$ to handle all differences in any given $D_p^{(k)}$, we may cope with these primes sparing an interval of length at most $(1+o_k(1))k^2n/(\ln k)^2 = o_k(k^2n/\ln k)$ to that end. 
		
		Let $\mathbb{P}$ denote the set of all prime numbers. We partition the remaining primes $p$, with $k/\ln k<p\leqslant k$, into subsets $P_i:=\{p\in \mathbb{P}: k/(i+1)<p\leqslant k/i\}$, where $i=1,2,\ldots,\ell=\lfloor \ln k \rfloor$. For each $i\leqslant\ell$ we will now pack in a separate short interval $J_i$ all $B_d$ with $d\in\bigcup_{p\in P_i} D_p^{(k)}$, thus completing the construction of packing of all $B_d$ with $d\leqslant k$.
		
		For $i=1$ we have $D_p^{(k)}=\{p\}$ for every $p\in P_1$, and by The Chineese Reminder Theorem we may pack all $B_p$ with $p\in P_1$ greedily one after another so that the intersection of intervals spanned by the embedded copies of $B_p$ and $B_q$ for any two consecutive $p<q$ in $P_1$ has length at least $(1-o_k(1))\min\{pq,pn/2\}$. Consequently, we may pack $B_p$ with $p\in P_1$ in an interval $J_1$ of length
		\begin{equation}\label{J1}
		(1+o_k(1))\left(\left(\sum_{p\in P_1\cap [n/2]} pn - \sum_{p\in P_1\cap[n/2]}p^2\right)
		+\frac{1}{2}\sum_{p\in(n/2,k]\cap\mathbb{P}}np\right).
		\end{equation}

		Let us now fix any $i\in[2, \ell]$. We will show that  
		we may pack all sets $B_d$ with $d\in\bigcup_{p\in P_i} D_p^{(k)}$ in an interval $J_i$ of length 
			\begin{equation}\label{Ji}
			(1+o_k(1))\left(\sum_{p\in P_i} ipn - \sum_{p\in P_i}ip^2\right).
			\end{equation}
			To that end we will first, similarly as in the case of $i=1$, pack all $B_{ip}$ with $p\in P_i$ in an interval of the desired length -- note this way we will handle only the sets $B_d$ with the largest diameters for the given $i$. We will thus have to be careful 
			in order to be later able to
complete the packing -- 
by disjointly embedding 
the remaining sets of our interest within the same interval.

			Consider any two consecutive primes $p<q$ in $P_i$. Since integers $ip$ and $q$ are coprime, we may pack $B_{ip}$ and $B_{iq}$ in an interval $I=[a,b]$ of length at most $np+nq-(1-o_k(1))ipq$ with
			a shifted copy $C_p=\{a=a_1,a_2,\ldots,a_n\}$ of $B_{ip}$ disjoint with a shifted copy $C_q=\{b_1,b_2,\ldots,b_n=b\}$ of $B_{iq}$ so that there
			is a set $R_q\subseteq \{0,1,\ldots,q-1\}$ of $i$ distinct residues
			such that $a_j\not\equiv r\pmod q$ for every $a_j\geqslant b_1$ and each $r\in R_q$.
(The argument follows almost the same lines as the one used in Step~3 of the proof of Theorem~\ref{t11} to provide a similar overlapping packing, but with $m=(pi-i^2)q$ and residues $r=(-j'pi)\pmod q$ for $j'\in[i]$; we omit further details.) By an additional slight shift, if necessary, we may also assume that 
\begin{equation}\label{b1an}
b_1-a_n<ipq-2p\leqslant (k-2)p\leqslant (n-2)p.  
\end{equation}
Note we may further glue 
such packings for all pairs of consecutive primes in $P_i$  (by identifying copies of the same progressions), thus providing a packing of all $B_{ip}$, $p\in P_i$, with two-sided overlappings (of lengths $(1-o_k(1))ipq$ for any two consecutive primes $p,q$ in $P_i$), in an interval $J_i$ of the desired length, bounded above  by~(\ref{Ji}).

			It remains to complete the packing by embedding the remaining sets $B_d$ with $d\in\bigcup_{p\in P_i} D_p^{(k)}$ in $J_i$.
			For every $q\in P_i$ we proceed as follows. 
			Adapting the notation above, suppose $C_q=\{b_1,b_2,\ldots,b_n=b\}$ is the embedding of $B_{iq}$ in $J_i$ (one may think of shifting entire $J_i$ so that this holds).
			We then embed $B_q,B_{2q},\ldots,B_{(i-1)q}$ in $J_i$ by positioning a
			copy of each $B_{mq}$ so that it starts at 
			a point $c_m\in (b_1,b_1+q)$ 
			which is equivalent modulo $q$ to a different $r_m\in R_q$, provided above, i.e., $c_m\equiv r_m \pmod q$ and $r_m\neq r_{m'}$ for $m\neq m'$.
			This way the resulting embeddings of all $B_{mq}$ with $m\in [i]$ in $J_i$ are obviously disjoint. Moreover, by the choice of $R_q$, these are also disjoint with the embedding of $B_{ip}$ for $p$ directly preceding $q$ in $P_i$. Note additionally that for every $m\in[i-1]$, the last element of the embedding of $B_{mq}$ is smaller than $b_1+q+(n-1)(i-1)q=b_n-(n-2)q$, 
			and hence the condition~(\ref{b1an}) guarantees that the embedding of $B_{mq}$ will be disjoint with embeddings of all remaining sets $B_d$ with $d\in\bigcup_{p\in P_i} D_p^{(k)}$ in $J_i$.

			Consequently, for $k\leqslant n/2$, by summing up~(\ref{J1}) and~(\ref{Ji}) for $i\geqslant 2$, we conclude that all sets $B_d$ with $d\leqslant k$ can be packed in an interval of length
		$$(1+o_k(1))\left(\sum_{i\in [\ell]}\sum_{p\in P_i}ipn -\sum_{i\in [\ell]}\sum_{p\in P_i}ip^2\right),$$ 
		whereas by Lemma~\ref{l21} (where by $f\sim g$ we mean $f=(1+o_k(1))g$),
		$$
		\sum_{i\in [\ell]}\sum_{p\in P_i}ipn 
		\sim n\sum_{i\in [\ell]} i\left(\frac{\left(\frac{k}{i}\right)^2}{2\ln k}-\frac{\left(\frac{k}{i+1}\right)^2}{2\ln k}\right)
		\sim \frac{nk^2}{2\ln k} \sum_{i\in [\ell]}\frac{1}{i^2}
		\sim \frac{nk^2}{2\ln k} \cdot \zeta(2)
		$$
		and
		$$
		\sum_{i\in [\ell]}\sum_{p\in P_i}ip^2 
		\sim \sum_{i\in [\ell]} i\left(\frac{\left(\frac{k}{i}\right)^3}{3\ln k}-\frac{\left(\frac{k}{i+1}\right)^3}{3\ln k}\right)
		\sim \frac{k^3}{3\ln k} \sum_{i\in [\ell]}\frac{1}{i^3}
		\sim \frac{k^3}{3\ln k} \cdot \zeta(3).
		$$
		Finally, for $k>n/2$, we may bound the length of the shortest interval admitting packing of all progressions $B_d$ with $d\leqslant k$ by the same estimate as for $k\leqslant n/2$ above 
		plus the following, which accounts for alteration of the formula in~(\ref{J1}) for primes $p>n/2$: 
		\begin{eqnarray}
			&&(1+o(1))\left(\frac{1}{2}\sum_{p\in(n/2,k]\cap\mathbb{P}}np-\sum_{p\in(n/2,k]\cap\mathbb{P}}(pn-p^2)\right) \nonumber\\
			&=& (1+o(1))\left(\sum_{p\in(n/2,k]\cap\mathbb{P}}p^2-\frac{n}{2}\cdot \sum_{p\in(n/2,k]\cap\mathbb{P}}p\right) \nonumber\\
			&=& (1+o(1))\left(\left(\frac{k^3}{3\ln k}-\frac{\left(\frac{n}{2}\right)^3}{3\ln k}\right)
			-\frac{n}{2}\cdot \left(\frac{k^2}{2\ln k}-\frac{\left(\frac{n}{2}\right)^2}{2\ln k}\right)\right) \nonumber\\
			&=& (1+o(1))\left(\frac{1}{3}\cdot\frac{k^3}{\ln k}
			+\frac{1}{48}\cdot\frac{n^3}{\ln k}
			-\frac{1}{4}\cdot\frac{nk^2}{\ln k}\right). \nonumber
		\end{eqnarray}	
	\end{proof}
	
	It seems that one of the key steps towards providing optimal multiplicative constants in Theorems~\ref{t11}, \ref{t31} and~\ref{t41} above should 
	rely on 
	understanding and designing appropriate packings of arithmetic progressions with prime differences.

\section{Bound for $M_k(n)$ when $k\rightarrow \infty$; Proof of Theorem \ref{t13}}
	
	In this section we prove Theorem~\ref{t13}.  To do so we first pack almost all arithmetic progressions from the collection $\{B_1,B_2,\ldots, B_k\}$ into a \emph{cyclic} interval of 
size $nk$, which will give us a packing of these progressions into $[m]$, with $m\leqslant 2nk$. The rest of the progressions can be packed in another interval of length $nk$ (in fact, even a 
slightly shorter interval of length $nk-k+o(n^2 k)$), which jointly give the result as stated in Theorem~\ref{t13}.

We will use the known result about nearly perfect matchings in nearly-regular uniform hypergraphs with small co-degrees. The required statement, proved  in \cite{FR} and \cite{PS} following the initial
ideas of R\"odl \cite{Ro}, is stated in the following lemma. Recall that the maximum \emph{co-degree} of a hypergraph is the maximum number of edges containing a fixed pair of vertices.
\begin{lemma}[\cite{FR}, \cite{PS}]
\label{lfr}
For every fixed $t$ and $\mu>0$ there is a $D_0$ and $\delta>0$ such that any $t$-uniform hypergraph on $N$ vertices in which the degree of every vertex is between
$(1-\delta)D$ and $(1+\delta)D$, where $D \geqslant D_0$, and the maximum co-degree is at most $\delta D$, contains a matching covering all but at most $\mu N$ vertices.
\end{lemma}
	
In the result below a \emph{cycle} of \emph{length} $r$ is simply the cyclic group $\mathbb{Z}_r$. Let $\mathcal{G}_n(k)=\{B_1,B_2,\ldots,B_k\}$ be the family of arithmetic progressions $B_d=\{d,2d,\ldots,nd\}$, with $d\in [k]$.
	
\begin{theorem}\label{t51}
Let $n$ be a fixed positive integer. For every $\epsilon>0$ there exists $k_0$ such that for every $k\geqslant k_0$, there is a set of at least $(1-\epsilon)k$ distinct arithmetic progressions from 
the family $\mathcal{G}_n(k)$ that can be packed into a cycle of length $nk$.
\end{theorem}
\begin{proof} 
	Put $t=n+1$, $\mu=\epsilon/(n+1)$, $N=nk+k$, and consider the following $t$-uniform hypergraph on the set of $N$ vertices $[k] \cup \mathbb{Z}_{nk}$. For each  $d \in [k]$ there are $kn$ edges, where the first vertex of each of these edges is $d \in [k]$ and the other $n$ vertices are a cyclic shift of the progression $B_d$ in $\mathbb{Z}_{nk}$. Note that this hypergraph is $D=kn$-regular. It is also easy to check that the maximum co-degree in it is smaller than $n^3$. As $n$ (and hence $t$) are fixed, if $k$ is sufficiently large, then by Lemma \ref{lfr} the hypergraph contains a matching covering all vertices besides at most $\epsilon k
(=\mu N)$ of them. This gives the required progressions.
\end{proof}	

The assertion of Theorem \ref{t13} follows quickly from the last result. As mentioned above, the cyclic packing gives a packing of almost all progressions into $[m]=[2nk]$. The rest of the progressions can be packed in another interval of length $nk-k+o(n^2 k)$ using Lemma \ref{l23}.

\section{Concluding remarks and open problems}

\subsection{Error correcting codes and Ferres shapes}
The packing problem considered here is similar in spirit to many packing problems in discrete mathematics. A prominent example is the main combinatorial problem in Coding Theory. 
Here the objective is to pack shifts of $m$ \emph{Hamming balls} of binary vectors of a given 
length $n$ and radius $r$ in the Hamming cube $\{0,1\}^n$, where the shifts are performed in the group  $\mathbb{Z}_2^n$. The existence of such a packing with no overlap is equivalent to the existence of a binary \emph{error correcting code} with $m$ codewords, length $n$, and distance at least
$2r+1$.  See, e.g., \cite{MS} and the references therein for the many relevant known results and open problems.

Another example of a problem of this type is discussed in \cite{ABS}. The main result there
is that the maximum fraction of all the $p(n)$ \emph{Ferrers shapes} corresponding to partitions of an integer $n$ that can be packed with no overlap in a rectangle of length $n \times p(n)$ is $O(1/\log n)$. Note that here the shifts are two-dimensional. This can be used to show
that the smallest length $t=t(n)$  of an $n$ by $t$ rectangle in which shifts of all
the $p(n)$ shapes can be packed with no overlap is at least $\Omega(p(n) \log n)$. Indeed, we can cover any $n$ by $cp(n) \log n$ rectangle by, say, $c \log n +1$ subrectangles,
each of height $n$ and length smaller than $p(n)$, so that consecutive rectangles have an overlap of a common $n$ by $n$ square. If the original rectangle contains all the Ferrers shapes, then each shape is fully contained in one of the smaller rectangles, which cannot cover a fraction of at least $1/(c \log n+1)$ of all shapes for an appropriately chosen constant $c$.
Therefore $t(n) > c p(n) \log n$ for some absolute constant $c>0$. 

It is not too difficult to show that this is tight up to the absolute constant $c$.  Indeed,
by the Hardy-Ramanujan asymptotic formula for the number of shapes of size
$n$, see, e.g., \cite{An}, 
\begin{equation} 
	\label{e31}
	p(n) =(1+o(1)) \frac{e^{C\sqrt{n}}}{4n\sqrt{3}}, 
\end{equation}
where $C=\pi \sqrt{\frac{2}{3}}$, and the $o(1)$-term tends to $0$ as $n$ tends to infinity. This easily implies that the number of shapes in which the first row $x_1$ or the first column 
$y_1$ is larger than $A \sqrt n \log n$ is less than $p(n)/ n$ for an appropriate  absolute constant $A$. (To see this note that the number of shapes with first row
of length $x_1$ is $p(n-x_1)$ and a similar formula holds for the first column.)
We can thus first omit all these exceptional shapes, pack all the others in an $n$ by $O(p(n) \log n)$ rectangle using the packing along diagonals described in the final remark of \cite{ABS}, and complete the packing by appending one additional $n$ by $n$ square for each of the remaining exceptional shapes.

\subsection{Progressions with prime differences}
Our main result in Theorem~\ref{t11} shows that, somewhat surprisingly, the minimum length of an interval in which all the $n$ progressions $A_d$ with differences $d \in [n]$ can be packed with no overlap
is the same, up to a constant factor, as the length required for all prime differences $p \leqslant \sqrt n$. Moreover, a close look at the proof reveals that it implies the following.
\begin{proposition}
	\label{p31}
	For any $\epsilon>0$ there exists $\delta=\delta(\epsilon)>0$ so that the following holds. Let $F$ be the set of all differences $d\leqslant \sqrt{n}(\ln n)^3$ divisible by a prime $p$, with $\delta \sqrt n \leqslant p \leqslant \sqrt n$. Then,  
	$$
	m_{[n]-F}(n) \leqslant \epsilon \frac{n^{3/2}}{\ln n} (< \epsilon m(n)).
	$$
	Therefore, $(1-\epsilon) m(n) \leqslant m_F(n) (\leqslant m(n)).$
\end{proposition}
\begin{proof}
	The total length required for packing the progressions with differences considered in
	steps 1,2, and 4 in the proof of the upper bound of Theorem \ref{t11} is $o(n^{3/2}/\ln n)=o(m(n)).$
	The same holds for differences $d$ which are considered in step 3 and are divisible by a prime $p>\sqrt{n}$.
	By our argument in step 3, the contribution  from  the remaining differences from this step 
	which are
	divisible by a prime smaller than $\delta \sqrt n$ is at most $n$ times the number
	of these primes.
	The desired result thus follows by the Prime Number Theorem.  
\end{proof}
\subsection{Beyond arithmetic progressions with bounded diameter}
	The upper bound in Theorem \ref{t11} applies the arithmetic properties of the progressions, besides the fact that their sizes are $\lfloor  n/d \rfloor$, for $1 \leqslant d \leqslant n$. The sizes themselves (along with the diameter constraint) suffice to imply that the sets can be packed in an interval of length $\Theta(n^{3/2})$, and this is tight up to the hidden constant
in the $\Theta$ notation. This is proved in the following simple statement, in which we do not make a serious attempt to optimize the absolute constants.
\begin{proposition}
	\label{p32}
	Let $C_i$, $1 \leqslant i \leqslant n$, be an arbitrary collection of subsets of $[n]$, where for all admissible $i$, $|C_i|=\lfloor \frac{n}{i} \rfloor$. Then there are shifts $r_i+C_i$ that can be packed in an interval
	of length at most $(1+1/e+o(1))n^{3/2}$. This is tight up to the constant factor, that is, there are such sets $C_i$ that cannot be packed in an interval of length smaller than $(1+o(1))n^{3/2}/2$.
\end{proposition}
\begin{proof}
	The first $\sqrt n$ sets $C_i$ can clearly be packed in an interval of length $n^{3/2}$ by shifting each $C_i$ to a separate interval of length $n$. Starting  then with an empty new interval of length 
	$bn^{3/2}$ we pack shifts of the sets $C_i$ for $i > \sqrt n$ one by one, in order, using Lemma \ref{l23}. When trying to embed the set $C_x$, the total size of the sets embedded so far (in this new interval) is
	at most 
	$$
	\sum_{\sqrt n < i<x} \frac{n}{i} =
	(1+o(1)) n (\ln x - \ln {\sqrt n}).
	$$
	(We assume here that, say, $x \geqslant 1.1 \sqrt n $, it is easy to see that there is no problem to embed the earlier sets.)
	
	By Lemma \ref{l23} the embedding can be performed provided for every relevant $x$,
	\begin{equation}
		\label{e31}
		(1+o(1)) n (\ln x - \ln {\sqrt n}) \cdot \frac{n}{x} < b n^{3/2}-n.
	\end{equation}
	The function $f(x)=\frac{\ln x -\ln \sqrt n}{x}$ attains its maximum for $x > \sqrt n$ at $x=e \sqrt n$, where its value is $1/(e \sqrt n)$. Therefore, the inequality in (\ref{e31}) holds for any fixed $b>1/e$ (and sufficiently large $n$). This establishes the first part of the proposition.
	
	To see that it is tight up to a constant factor consider an example in which each of the first $\sqrt n/2$ sets $C_i$ contains the interval	$[\sqrt n]=\{1,2, \ldots ,\sqrt n\}$ as well as the arithmetic progression $\{\sqrt n, 2 \sqrt n, 3 \sqrt n, \ldots,n\}$ (assume, for simplicity, that $\sqrt n $ is an integer.) Then in any disjoint packing of shifts of these $\sqrt n/2$ sets $C_i$, no two intervals of length $n$ containing the shifted $C_i$ can have a nonempty overlap. This completes the proof of the proposition.
\end{proof}
The constants in the upper bound and in the lower bound in the last proposition can be improved by being a bit more careful. In the upper bound the constant can be improved to $\frac{2}{\sqrt e}+o(1)$. This is done
by embedding each of the first $\lfloor \frac{\sqrt n}{\sqrt e} \rfloor$ sets $C_i$ in a separate interval of length $n$, and then by embedding the remaining sets $C_i$, in order, one by one, using Lemma \ref{l23}, in another interval of length $(\frac{1}{\sqrt e}+o(1))n^{3/2}$. The constant in the lower bound can be improved to $\ln 2+o(1)$. To do so, add to the initial sets $C_i$ for $i \leqslant \sqrt n/2$ described above additional
sets $C_j$ for $\sqrt n/2 <j \leqslant \sqrt n$, where for each such $j$, $C_j$ contains the interval
$[\sqrt n]=\{1,2, \ldots ,\sqrt n\}$ together with the arithmetic progression $\{\sqrt n, 2 \sqrt n, 3 \sqrt n, \ldots, (n/j-\sqrt n) \sqrt n\}$.

\subsection{Conjectures}
	It may be interesting to find an asymptotic formula for $m(n)$ and determine a constant $b$ so that 
	$m(n)=(1+o(1))b \frac{n^{3/2}}{\ln n}$. This seems to require additional ideas. The same remark applies to other functions we considered. Below we state a list of conjectures regarding this issue.
	
	\begin{conjecture}\label{c61}
	$m(n)\sim \frac{4}{3}\cdot \frac{n^{3/2}}{\ln n}$.
	\end{conjecture}
	We also suspect $1/6$ should approximate the multiplicative factor for $M_n(n)$, cf. Theorem~\ref{t41}.
	\begin{conjecture}\label{c62}
		$M(n)\sim \frac{1}{6}\cdot \frac{n^{3}}{\ln n}$.
	\end{conjecture}
	Our apparently much weaker conjectures concern the set of prime numbers exclusively, whose understanding seems to be one of the key objectives towards solving more general settings.
	Denote by $\mathbb{P}(x)$ the set of prime numbers not exceeding $x$.
	\begin{conjecture}\label{c63}
		$m_{\mathbb{P}(\sqrt{n})}(n)\sim \frac{4}{3}\cdot \frac{n^{3/2}}{\ln n}$.
	\end{conjecture}
	\begin{conjecture}\label{c64}
		$M_{\mathbb{P}(n)}(n)\sim \frac{1}{6}\cdot \frac{n^{3}}{\ln n}$.
	\end{conjecture}
	The two problems above, in Conjectures~\ref{c63} and~\ref{c64}, seem to be closely related.
	Note also that the lower bounds in these follows from the proofs of the lower bounds in Theorem~\ref{t11} and Theorem~\ref{t41}, respectively.
	
	As Conjecture~\ref{c61}, if true, resolves asymptotically also the issue of $m_k(n)$ for $k\geqslant\sqrt{n}$, in view of Theorem~\ref{t31}, the only open case would concern $k<\sqrt{n}$ which are not $o(\sqrt{n})$.
	\begin{conjecture}
		For every $k$ satisfying $\Omega(\sqrt{n})=k<\sqrt{n}$, we have $$m_k(n)\sim \left(1-\frac{k^2}{3n}\right)\frac{k}{\ln k}\cdot n.$$
	\end{conjecture}
	The following seem to be plausible in turn for $M_k(n)$ with $k<n$, in view of Theorem~\ref{t41}.
	\begin{conjecture}
	If $k\ll n$, then $$M_k(n)\sim \frac{k^2}{2\ln k}\cdot n,$$
	while otherwise, $$M_k(n)\sim \frac{k^2}{2\ln k}\cdot n-\frac{k^3}{3\ln k}.$$
	\end{conjecture}

The last conjecture concerns packing of $k$ arithmetic progressions of the form $B_d=\{d,2d,\ldots,nd\}$, for $d=1,2,\ldots,k$, with the length $n$ fixed. We suspect that perhaps the trivial lower bound of $nk$ is asymptotically correct. 
\begin{conjecture}
For every fixed $n$ and $k$ tending to infinity, we have 
$$M_k(n)=(1+o(1))nk.$$
\end{conjecture}

\end{document}